\documentclass{amsart}
\usepackage{float}
\usepackage{graphicx}
\usepackage{amscd}
\usepackage{amsmath}
\usepackage{amsfonts}
\usepackage{amssymb}
\usepackage{fancyhdr}
\usepackage{lipsum}
\usepackage[]{mdframed}
\usepackage{newfloat}
\usepackage{caption}
\usepackage{fullpage}
\usepackage{amsfonts}
\usepackage{url}
\usepackage{graphicx}
\usepackage{hyperref}
\usepackage{graphpap}
\usepackage{amsmath}
\usepackage{setspace}
\usepackage{graphics}
\usepackage{graphicx}
\usepackage[T1]{fontenc}
\usepackage{textcomp} 
\usepackage{mathptmx}

\theoremstyle{plain}
\newtheorem{theorem}{Theorem}[]
\newtheorem{corollary}[theorem]{Corollary}

\newtheorem{proposition}[theorem]{Proposition}
\theoremstyle{definition}
\newtheorem{example}[theorem]{Example}

\newtheorem{definition}[theorem]{Definition}

\newtheorem{remark}[theorem]{Remark}
\theoremstyle{remark}

\numberwithin{equation}{section}

\email{shafiq\_ur\_rahman2@yahoo.com, shafiq@cuiatk.edu.pk}
\email{raheelatahir.cui@gmail.com}
\email{farhatnoor.cuiatk@gmail.com}

\begin{document}

\title[Some Properties of Order-Divisor Graphs of  Finite Groups]{ Some Properties of Order-Divisor Graphs of  Finite Groups}


\author{Shafiq ur Rehman, Raheela Tahir, Farhat Noor }

\address{}
\address{{\em (Rehman, Tahir, and Noor)  COMSATS University Islamabad, Attock Campus,
Pakistan.}}

\thanks{}

\begin{abstract}
This article investigates the properties of order-divisor graphs associated with finite groups. An order-divisor graph of a finite group is an undirected graph
in which the set of vertices includes all elements of the group, and two distinct vertices with different orders are adjacent if the order of one vertex
divides the order of the other. We prove some beautiful results in order-divisor graphs of finite groups. The primary focus is on examining the girth,
degree of vertices, and size of the order-divisor graph.
 In particular, we provide a comprehensive description of these parameters for the order-divisor graphs of finite cyclic groups and dihedral groups.
\end{abstract}

\maketitle


\section[Introduction]{Introduction}
Algebra and graph theory have a wide range of real-world applications. In group theory, for example, the study of symmetry finds relevance in areas such as crystal structures, artistic designs, and musical compositions. Theoretical computer science relies on different algebraic structures, while lattice theory aids in the formulation of semantic models. Universal algebra is crucial for the formalization of data types, and category theory forms the basis of type theory.

Graph theory, which focuses on the study of graphs made up of vertices and edges, is extensively used in computer science disciplines, including data mining, image processing, clustering, and network analysis. Graph coloring techniques are vital for addressing challenges in resource distribution and task scheduling. Additionally, the notions of paths, walks, and circuits in graph theory provide solutions to complex problems like the traveling salesman, database architecture, and resource network optimization.

Cayley's Theorem in group theory serves as an inspiration for algebraic combinatorics.
Group action involves a dynamic interaction with an object, resulting in the partitioning of its elements into distinct orbits.
Analyzing the structure and number of these orbits leads to significant
 combinatorial insights. Both algebraic combinatorics and graph theory have important applications in computer science,
 including automata theory, complexity theory, and Polya enumeration theory.

The relationship between groups and graphs is inspired by concepts such as the Cayley graph of a group and the automorphism group of a graph.
The Cayley graph $Cay(X:G)$ of a finite group $G$ with a generating set $X$ is a directed graph where each element of $G$ represents a vertex.
There is an arc from a vertex $a$ to a vertex $b$ iff there exists an $x \in X$ such that $ax = b$.
The automorphism group of a graph consists of all possible graph automorphisms, which form a group under function composition.
Relevant literature on these topics includes sources \cite{FB, F, L, L2, MPA}.

Numerous approaches can be used to connect a group  to an undirected graphs. A significant amount of
scholarly work is dedicated to investigating graphs that correspond to finite groups,  including  commuting graphs \cite{AMRR, ASH, BBPR, R, P},
non-commuting graphs \cite{AAM},
intersection graphs \cite{AHM}, prime graphs \cite{BC1, K, Lu}, power graphs \cite{C1, C2, MAN}, inverse graphs \cite{ZA}, equal-square graphs \cite{R1},
and the order-divisor graphs [21]. The order-divisor graph of a finite group $G$,  denoted by $OD(G)$,    is a graph whose
vertex set consist of all elements of $G$ such that two distinct vertices with different orders are adjacent provided that order of one vertex
divides the other.  It represents a novel approach to depicting groups through graphs.

To establish a foundational understanding, we recall the following concepts:
The collection of all symmetries of a regular polygon with $n$ sides (where $n \geq 3$) forms a group called the
dihedral group, denoted by $\mathcal{D}_n$, with order $2n$. The dihedral group $\mathcal{D}_n$ can be defined
using the generators $\alpha$ and $\beta$ with the relation
$\alpha^n =e$,  $\beta^2 =e$, and  $(\alpha \beta)^2 = e$.
A cyclic group is characterized by being generated by a single element.
A finite cyclic group of order $n$ is isomorphic to the group of residue classes modulo $n$, denoted by  $\mathbb{Z}_n$, under addition.
The group of units  $U(n)$, consists of elements $\bar{x} \in \mathbb{Z}_n$ for which $(x,n) = 1$.
For each prime number $p$,  a group is termed as $p$-group if each of its element has order a power of $p$.
The exponent of a group $G$ is the least positive integer $k$ for which $g^k = e$ $\forall g \in G$.
A group $G$ is termed as elementary abelian group (or an elementary abelian
$p$-group) if it is a $p$-group with
exponent $p$, meaning $x^p = e$ for every $x \in G$. A finite elementary abelian group is isomorphic to ${\mathbb{Z}_{p}}^n$ (where
$p$ is a prime and  $n$ is a positive integer). A simple graph is an undirected graph without loops or parallel edges.
A complete bipartite graph is one where each vertex in one partite set
is adjacent to every vertex in the other partite set.
The star graph $S_n$ of order $n$ is a tree in which one vertex has degree $n-1$ and the remaining vertices have degree
$1$, i.e., $S_n$ being isomorphic to $K_{1,n}$.

In this paper, the following outcomes are achieved: Let $G$ denotes the group.
  $OD(G)$ is a star graph  iff all elements of $G$, excluding the identity, have prime orders iff
$OD(G)$ contains no cycles iff $OD(G)$ is bipartite (Theorem \ref{2}).  If  $|G|$  is prime then $OD(G)$ form a star graph
(Corollary \ref{a}). $OD(G)$ is a path graph  iff $|G|$ $=$ $2$ or $3$ (Corollary \ref{b}). 	
If $OD(G)$ contains a cycle, then$g(OD(G))$ $=$ $3$ (Theorem \ref{3}). 	
$G$ contains an element with a composite order iff $g\big(OD(G)\big)$ is $3$ (Theorem \ref{1}). If $G$ is cyclic, then $g(OD(G))=3$ iff $G$ has composite order
(Corollary \ref{c}). If $G$ is a finite abelian group such that $|G|$ is divisible by two different primes, then $g\big(OD(G)\big) = 3$ (Corollary \ref{d}).
If  $|G| = pq$, where $p$, $q$ are primes, then $g\big(OD(G)\big) = 3$ iff $G$ is cyclic (Corollary \ref{e}).
For the dihedral group $\mathcal{D}_n$,  $OD(\mathcal{D}_n)$  has girth $3$ iff $n$ is composite (Corollary \ref{f}).
If  $E$ and $F$ are  groups then $g\big(OD(E\oplus F)\big) = 3$ iff at least one of the following is true:
(a) $E$ has an element of composite order, (b) $F$ has an element of composite order, (c) There are elements $x \in E$ and $y \in F$  whose orders are
 are different primes (Corollary \ref{g}). If $E$ and $F$ are cyclic groups then $g\big(OD(E\oplus F)\big) = 3$ iff at least one of the following is true:
	(a) $|E|$ is composite, (b) $|F|$ is composite, (c) $|E|$ and $|F|$ are distinct primes (Corollary \ref{h}). If $x \in \mathbb{Z}_n$
has order $m$ then $deg(x) = m - 2\phi(m) +  \sum_{\lambda \rvert \frac{n}{m}}\phi(\lambda m)$(Theorem \ref{4}).
For every  prime $p$ and each positive integer $k$,
if $x \in \mathbb{Z}_{p^k}$ has order $p^i$,
then $deg(x)=p^k + p^{i-1} - p^i$ for  $i\geq1$ and $deg(x)=p^k-1$ for $i = 0$ (Corollary \ref{6}). For every prime $p$ and each positive integer $k$,
$\sum_{x\in \mathbb{Z}_{p^k}}	deg(x)$ $=$ $\frac{2 p^{2k}-2}{p+1}$ (Corollary \ref{7}).
For each positive integer $k$, $\sum_{x\in \mathbb{Z}_{2^k}}	deg(x)$ $=$ $\frac{ 2^{2k+1}-2}{3}$  (Corollary \ref{9}).
For  every prime $p$ and each positive integer $k$,  $\sum_{x\in\mathbb{Z}_{p^k}} o(x) = \frac{p^{2k+1}+1}{p+1}$  (Corollary \ref{10}).
For each positive integer $k$,  $\sum_{x\in\mathbb{Z}_{2^k}} o(x)=\frac{2^{2k+1}+1}{3}$ (Corollary \ref{11}).
For each positive integer $k$, $1+\sum_{x\in \mathbb{Z}_{2^k}}deg(x)$ $=$ $\sum_{x\in \mathbb{Z}_{2^k}} o(x)$  (Corollary \ref{12}).
The  size  of  $OD\big(\mathbb{Z}_n\big)$ is given by  $|E\big(OD(\mathbb{Z}_n)\big)|=\frac{1}{2}\big[\sum_{m \rvert n}\big(m+\sum_{\lambda|\frac{n}{m}}\phi
(\lambda m)-2\phi(m)\big)\phi(m)\big]$ (Theorem \ref{13}). For each prime $p$ and the positive integer $k$, the size of $OD(\mathbb{Z}_{p^k})$ is given by
$\frac{p^{2k}-1}{p+1}$ (Corollary \ref{14}). If $n \geq 3$ is an odd  integer and $x \in \mathcal{D}_n$ has order $m$,  then
$deg(x)=	m - 2\phi(m) + \sum_{\lambda \rvert \frac{n}{m}}\phi(\lambda m)$ for  $m\textgreater 2$, $deg(x)=2n-1$ for  $m = 1$, and
$deg(x)=1$ for $m = 2$ (Theorem \ref{5}). If $n \geq 3$ is an even  integer and $x \in \mathcal{D}_n$ has order $m$, then
$deg(x)=m - 2\phi(m) + \sum_{\lambda \rvert \frac{n}{m}}\phi(\lambda m)$ for  $m$ is odd,
$deg(x)=\sum_{\lambda \rvert \frac{n}{2}}\phi(2 \lambda )$ for $m = 2$,
$deg(x)=n+m - 2\phi(m) + \sum_{\lambda \rvert \frac{n}{m}}\phi(\lambda m)$ for   $m \textgreater 2$ and even, and
$deg(x)=2n-1$ for if $m=1$ (Theorem  \ref{8}). If $n \geq 3$ is an odd  integer, then
the size  of  $OD\big(\mathcal{D}_n\big)$ is given by
$|E(OD(\mathcal{D}_n))|=\frac{1}{2}\big[(3n-1+\sum_{\substack {m|n, m>1}}\big(m-
2\phi(m)+\sum_{\lambda|\frac{n}{m}}\phi(\lambda m)\big)\phi(m)\big]$ (Theorem \ref{15}).
If $n \geq 3$ is an even  integer,  then
the size  of  $OD\big(\mathcal{D}_n\big)$ is given by
$|E(OD(\mathcal{D}_n))|=\frac{1}{2}\big[2n-1+(n+1)\sum_{\lambda \rvert \frac{n}{2}}
\phi(2 \lambda )+\sum_{\substack {m > 1, m \mid n\\ m~is~odd}}\big(m-2\phi(m)+
\sum_{\lambda|\frac{n}{m}}\phi(\lambda m)\big)\phi(m)+ \sum_{\substack {m > 2, m \mid n\\ m~is~even}}\big(n+m-
2\phi(m)+\sum_{\lambda|\frac{n}{m}}\phi(\lambda m)\big)\phi(m)\big]$ (Theorem \ref{16}).
In the final section, we provide PARI/GP codes in the PARI/GP software
\cite{P} to calculate the degree of each vertex and the size within the order-divisor graph of
the finite cyclic group $\mathbb{Z}_n$ and the dihedral group $\mathcal{D}_n$.

Unless otherwise specified, all groups and graphs discussed in this paper are assumed to be finite. Any concepts or results not detailed here can be understood as
standard, as referenced in \cite{G}, \cite{KHR} and \cite{W}.
\section{Properties of Order-divisor graph $OD(G)$}

According to \cite{RBIK}, the order-divisor graph of a finite group $G$ is a graph where the vertices represent the elements of $G$ such that two
distinct vertices $a$ and $b$ are adjacent if their orders are unequal and the order of one vertex divides the order of the other.
This graph is symbolized by $\b{OD}(G)$.
\\

Some immediate consequences about the order-divisor graphs are listed below .
\begin{enumerate}
\item[(a)] The graph $OD(G)$ is a simple graph.

\item[(b)] The radius of $OD(G)$ is $1$, and its diameter is $2$.

\item[(c)] For $|G|>2$, $OD(G)$ cannot be a complete graph.

\item[(d)] $OD(G)$ cannot form a cycle.
	
\end{enumerate}
We summarize here some basic facts related to $OD(G)$ from \cite{RBIK}.
\begin{proposition}{\em (\cite[Theorem 3, Corollary 4]{RBIK})}
The following statements are true for a group $G$.
\begin{enumerate}
\item[(a)]$OD(G)$ is a star graph iff each non-identity element in $G$ has prime order.

\item[(b)] If $G$ is abelian, then $OD(G)$ is star iff $G$ is elementary abelian.
\end{enumerate}
\end{proposition}
\begin{proposition}{\em (\cite[Corollary 6, Corollary 7, Theorem 9, Corollary 13]{RBIK})}
\\The following assertions hold for the groups $\mathbb{Z}_n$, $U(n)$ and $D_n$.
\begin{enumerate}
\item [(a)]The graph $OD(U(n))$ is a star graph $S_{\phi(n)}$ iff $n$ divides $24$.
\item [(b)] The graph $OD(\mathbb{Z}_n)$ is a star graph iff $n$ is a prime number.
\item [(c)] The graph $OD(\mathcal{D}_n)$ is a star graph $S_{2n}$ iff $n$ is a prime,\\ where $\mathcal{D}_n$ (for $n \geq 3$) represents the dihedral group of order $2n$.	
\item[(d)]  The chromatic number of $OD(\mathbb{Z}_n)$ is $n + 1$.
\end{enumerate}
\end{proposition}
Recall \cite{W} that a graph is called bipartite if a set of graph vertices decomposed into
two disjoint sets such that no two graph vertices within the same set are adjacent.
A well known result about bipartite graph is that a graph is bipartite iff it contains no cycles of odd length.
Applying this result, we  determine  bipartite property for $OD(G)$. We describe some results about when $OD(G)$ is
a bipartite graph and also when $OD(G)$ is a star graph.
\begin{theorem}\label{2}
	For a group $G$, the following statements are equivalent:
\begin{enumerate}
	\item[(a)] $OD(G)$ is a star graph.
	\item[(b)] All elements of $G$, excluding the identity, have prime orders.
	\item[(c)] The graph $OD(G)$ contains no cycles.
	\item[(d)] The graph $OD(G)$ is bipartite.
\end{enumerate}
\end{theorem}
\begin{proof}
(a) $\iff$ (b): Apply \cite[Theorem 3]{RBIK}
\\
(a) $\iff$ (c) and (a) $\implies$ (d) are straight forward.
\\
(d) $\implies$ (a) Suppose $OD(G)$ is bipartite.
Then $OD(G)$ cannot have an odd cycle. Since identity vertex is associated to each vertex, therefore no two non-identity vertices could be adjacent.
 Hence $OD(G)$ is a star graph.
\end{proof}

\begin{corollary}\label{a}
If  $G$ is a group whose order is a prime number, then $OD(G)$ form a star graph.
\end{corollary}
\begin{remark}
The converse statement of Corollary \ref{a} is false. For example $OD(D_3)$ is a star graph but $|D_3|$ $=$ $6$.
\end{remark}
\begin{corollary}\label{b}
The graph $OD(G)$ is a path graph iff $|G|$ $=$ $2$ or $3$.
\end{corollary}
\begin{proof}
Let $OD(G)$ be a path graph. Then by Theorem \ref{2}, $OD(G)$ is a star graph. Therefore $OD(G)$ $=$ $P_2$ or $P_3$. Hence order of $G$ is $2$ or $3$.
\\
Conversely, suppose order of $G$ is $2$ or $3$. Then clearly $OD(G)$ is  $P_2$ or $P_3$.
\end{proof}
\section{The girth of the order-divisor graph $OD(G)$}
In this section, we will address the girth of order-divisor graph. Our main objective here is to calculate the girth by considering the order of group.   Moreover,
we determine the girth of order-divisor graph for both cyclic groups and the dihedral groups.
First result of this section shows that $OD(G)$ has girth $0$ or $3$.

\begin{theorem}\label{3}
If $OD(G)$ contains a cycle then $g(OD(G))$ $=$ $3$.
\end{theorem}
\begin{proof}
Suppose $OD(G)$ contains a cycle. Since the vertex corresponding to identity is adjacent to every other vertex,
therefore either $OD(G)$ is a star graph or it contains a cycle of
length $3$. Hence $g(OD(G)) = 3$.
\end{proof}

According to Theorem \ref{2}, $OD(G)$ is cycle-free iff it is a star graph. Consequently, when $|G|$ is prime then $OD(G)$ is cycle-free, i.e.,  $g(OD(G))=0$.
 Therefore, we are keen on exploring the girth in instances where $|G|$ is composite.

\begin{theorem}\label{1}
A group $G$ contains an element with composite order iff $g\big(OD(G)\big)$ is $3$.
\end{theorem}
\begin{proof}
Suppose $G$ has an element $x$ of order $n$, where $n$ is composite. Then $|G|=n\lambda$ for some integer $\lambda\geq1$.
Let $p$ be a prime divisor of $n$. By Cauchy's Theorem $G$ has an element $y$ of order $p$, cf.\cite{B}.
Hence $e-x-y-e$ is a cycle of length $3$ and therefore, $g\big(OD(G)\big)$ is $3$. Conversely, we suppose $g(OD(G))$ is $3$. Then by Theorem \ref{3},
$OD(G)$ can not be a star graph. We know that the vertex corresponding  to identity is adjacent to every other vertex, $i.e$., $e$ is central vertex, therefore
$\exists$ $x,y\in G$ with $o(x)\neq o(y)$ such that $o(x)|o(y)$ or $o(y)|o(x)$. Hence atleast one of $x$ or $y$ is composite.
\end{proof}
\begin{corollary}\label{c}
Let $G$ be a cyclic group. Then $g(OD(G))=3$ iff $G$ has composite order.
\end{corollary}
\begin{proof}
For each positive divisor $d$ of $|G|$, $G$ has some element of order $d$, cf. \cite[Theorem]{G}.  Apply Theorem \ref{1}.
	\end{proof}

\begin{remark}
It's noteworthy to highlight here that for abelian groups, the girth of the order-divisor graph equals zero
precisely when the group is isomorphic to ${\mathbb{Z}_{p}}^n$ for some prime $p$ and an integer $n \geq 1$. Therefore,
in the scenario of a composite-order abelian groups, the girth of the order-divisor graph is three exactly when the
group contains a pair of non-identity elements with different orders.
\end{remark}

\begin{corollary}\label{d}
	If $G$ is abelian group such that $|G|$ is divisible by two different primes,  then $g\big(OD(G)\big) = 3$.
\end{corollary}
\begin{proof}
	Let $G$ be a group with $|G| = n$ and let $n$ be divisible by two different primes $p$ and $q$.
By Cauchy's Theorem $G$ must have elements $x$, $y$ with $o(x)=p$ and $o(y)=q$. Since $G$ is abelian and $gcd(p,q)=1$, therefore $o(xy)=o(x)o(y)=pq$.
	Hence $G$ contains an element $xy$ of composite order. By using Theorem \ref{1}, $g\big(OD(G)\big)$ is $3$.
\end{proof}	

\begin{remark}
	In general, the converse of the preceding theorem is not true. For instance,
	$g\big(OD(\mathbb{Z}_8)\big)=3$ but $|\mathbb{Z}_8|=2^3$, cf. Figure \ref{f1}.
\end{remark}

\begin{center}
\begin{figure}[h]
\includegraphics[width=0.2\textwidth]{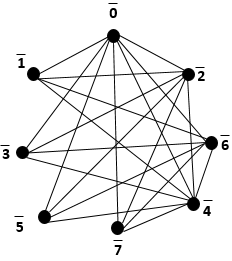}
\caption{$OD(\mathbb{Z}_8)$}
\label{f1}
\end{figure}

	\end{center}

\begin{corollary}\label{e}
	Let $G$ be a group such that $|G| = pq$, where $p$,$q$ are primes. Then $g\big(OD(G)\big) = 3$ iff $G$ is cyclic.
\end{corollary}
\begin{proof}
	Suppose $g\big(OD(G)\big)=3$. Then by Theorem \ref{1}, $G$ has some element $x$ of composite order. Since $o(x)\mid pq$, So $o(x)=pq$ and thus $G=<x>$.
 Conversely, Suppose $G$ is cyclic. Then the generator of $G$ has order $pq$ and hence by Theorem \ref{1}, $g(OD(G))=3.$
	\end{proof}

\begin{corollary}\label{f}
		The order-divisor graph $OD(\mathcal{D}_n)$ of the dihedral group
 $\mathcal{D}_n$ has girth $3$ iff $n$ is composite.

\end{corollary}
Next we determine the girth of order-divisor graph of external direct product $E \oplus F$ of groups $E$ and $F$, cf. \cite[Chapter 8]{G}.

\begin{theorem}\label{g}
	Let $E$ and $F$ be groups. The girth of $OD(E \oplus F)$ is $3$ iff at least one of the following is true:
	\begin{enumerate}
		\item[(a)] $E$ contains an element of composite order.
		\item[(b)] $F$ contains an element of composite order.
		\item [(c)] There are elements $x \in E$ and $y \in F$ whose orders are distinct prime numbers.
	\end{enumerate}
\end{theorem}

\begin{proof}
The order of an element $(a,b)$ in $E \oplus F$ is $lcm\big(o(a),o(b)\big)$, cf. \cite[Theorem 8.1]{G}. Then apply Theorem \ref{1}.
\end{proof}
\begin{corollary}\label{h}
Let $E$ and $F$ be cyclic groups. Then $g\big(OD(E\oplus F)\big) = 3$ iff at least one of the following is true:
\begin{enumerate}
\item[(a)] $|E|$ is composite
\item[(b)] $|F|$ is composite.
\item [(c)]$|E|$ and $|F|$ are distinct primes.
\end{enumerate}
\end{corollary}

\section{The Degree and Size within OD(G)}

\noindent This section is devoted to examine the degrees of vertices within  $OD(G)$ and the size of $OD(G)$. For the convenience and ease of the proofs, we define the following terminologies.
\begin{definition}
In each edge $xy$ in $OD(G)$, we call the vertex $x$ lower vertex and $y$ the upper vertex if $|x|$ divides $|y|$.
\end{definition}
\begin{definition}
	For each vertex $x$ in $OD(G)$ the lower degree of $x$, denoted by $Ldeg(x)$, is the number of lower vertices adjacent to $x$, similarly the upper degree of $x$, denoted by $Udeg(x)$, is defined to be the number of upper vertices adjacent to $x$.
\end{definition}

\subsection{The Degree and Size within $OD(\mathbb{Z}_n)$}$\,$
We describe here the degree of vertices and the size for the order-divisor graph of the finite cyclic groups.
We initiate our discussion with the aid of an example.

\begin{example}
Let us examine the order-divisor graph $OD(\mathbb{Z}_6)$ shown in Figure \ref{f2}.  We determine the degree of each vertex within $OD(\mathbb{Z}_6)$,
as shown in Table \ref{t1}.
\begin{center}
\begin{figure}[h]
	\includegraphics[width=0.2\textwidth]{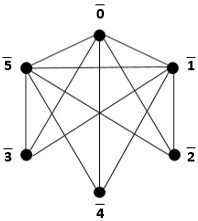}
\caption{$OD(\mathbb{Z}_6)$}
\label{f2}
\end{figure}
	\end{center}
\begin{center}
\begin{table}[!h]
	\begin{tabular}{|c|c|c|}
		\hline
		vertex & order & degree \\
		\hline
		$\overline{0}$ & $1$ & $5$ \\
		\hline
		$\overline{1}$ & $6$ & $4$\\
		\hline
		$\overline{2}$ & $3$ & $3$\\
		\hline
		$\overline{3}$ & $2$ & $3$\\
		\hline
		$\overline{4}$&$3$&$3$\\
		\hline
		$\overline{5}$&$6$&$4$\\
		\hline
	\end{tabular}
\caption{Degree $\&$ Order for $OD(\mathbb{Z}_6)$}
\label{t1}
\end{table}
\end{center}
\end{example}

\begin{remark}
It is quite difficult to compute directly the degree of each vertex in $OD(\mathbb{Z}_n)$, as in above table, for the large
values of $n$. Hence there is a need to find some explicit formula to determine the degree without actually drawing $OD(\mathbb{Z}_n)$. In this section, we introduce a general formula for calculating the degree of each vertex  in $OD(\mathbb{Z}_n)$. Using this formula, we also derive some intriguing insights related to the degrees of vertices  within $OD(\mathbb{Z}_n)$.
\end{remark}

\begin{theorem}\label{4}
	Let $x \in \mathbb{Z}_n$ has order $m$. Then\\
	\begin{equation*}
	 deg(x) = m - 2\phi(m) +  \sum\limits_{\lambda \rvert \frac{n}{m}}\phi(\lambda m),
	\end{equation*}
where $\phi$ represents the Euler's totient function.
\end{theorem}
\begin{proof}
We know that for each divisor $d$ of $n$, exactly $\phi(d)$ elements of order $d$ are contained in $\mathbb{Z}_n$, cf. \cite[Theorem 4.4]{G}. Hence
\begin{eqnarray*}
	Ldeg(x) &=& \sum\limits_{d \rvert m}\phi(d) - \phi(m)\\
	&=& m - \phi(m)
\end{eqnarray*}	
Also,
\begin{eqnarray*}
	Udeg(x) &=& \sum\limits_{\lambda m \rvert n}\phi(\lambda m) - \phi(m)\\
	&=& \sum\limits_{\lambda \rvert \frac{n}{m}}\phi(\lambda m) - \phi(m)
\end{eqnarray*}
Hence,
\begin{eqnarray*}
	deg(x) &=& Ldeg(x) + Udeg(x)\\
	&=& m - \phi(m) + \sum\limits_{\lambda \rvert \frac{n}{m}}\phi(\lambda m) - \phi(m)\\
deg(x)	&=& m - 2\phi(m) + \sum\limits_{\lambda \rvert \frac{n}{m}}\phi(\lambda m).
\end{eqnarray*}
\end{proof}
\begin{corollary}\label{6}
	Let $x \in \mathbb{Z}_n$ has order $m$. If $n = p^k$ and $m = p^i$, then\\
		\[deg(x)=
	\begin{cases}
		p^k + p^{i-1} - p^i, & \text{if $i\geq1$} \\
		p^k-1, & \text{if $i = 0$}
	\end{cases}\]
\end{corollary}
\begin{proof}
	From Theorem  \ref{4},\\
	\begin{eqnarray*}
		deg(x)	&=& m - 2\phi(m) + \sum\limits_{\lambda \rvert \frac{n}{m}}\phi(\lambda m)\\
		&=& p^i - 2\phi(p^i) + \sum\limits_{\lambda \rvert \frac{p^k}{p^i}}\phi(p^i\lambda)
	\end{eqnarray*}
When $i\geq1$,
		\begin{eqnarray*}
	deg(x)	&=& 2p^{i-1} - p^i + \sum\limits_{\lambda \rvert{p^{k-i}}}\phi(p^i\lambda)\\
		&=& 2p^{i-1} - p^i + \phi(p^i) + \phi(p^{i+1}) + \phi(p^{i+2}) +...+ \phi(p^k)\\
			&=& 2p^{i-1} - p^i + (p^i - p^{i-1}) + (p^{i+1} - p^i) + (p^{i+2}-p^i+1) + \cdots + (p^k-p^{k-1})\\
	&=& p^k + p^{i-1} - p^i.
\end{eqnarray*}
When $i = 0$,
\begin{eqnarray*}
	deg(x)&=& - 1 + 1 + \phi(p) + \phi(p^2) + \cdots + \phi(p^k)\\
	&=& p^k -1.
\end{eqnarray*}
\end{proof}

\noindent Next, we explore intriguing findings regarding the sum of degrees in the order-divisor graph and the sum
of the orders of each element in the associated group.\\

\begin{corollary}\label{7}
	For every prime $p$ and each positive integer $k$,
	$\sum\limits_{x\in \mathbb{Z}_{p^k}}	deg(x)$ $=$ $\frac{2 p^{2k}-2}{p+1}$.

\end{corollary}
\begin{proof}
We know that, for each divisor $d$ of $n$, $\mathbb{Z}_n$ contains exactly $\phi(d)$ elements that have order $d$. Therefore, by Corollary \ref{6},
\begin{eqnarray*}
	\sum\limits_{x\in \mathbb{Z}_{p^k}} deg(x) &=& (p^k-1) + (p^k+1-p) \phi(p) + (p^k+p-p^2) \phi(p^2) + \cdots + (p^k+p^{k-1}-p^k) \phi(p^k)\\
	&=& (p^k-1) + (p^k+1-p) (p-1) + (p^k+p-p^2) (p^2-p) + \cdots + (p^k+p^{k-1}-p^k) (p^k- p^{k-1})\\
	&=& (p^k-1) + (p^k-(p-1)) (p-1) + (p^k-(p^2-p)) (p^2-p) + \cdots + (p^k-(p^{k}-p^{k-1})) (p^k- p^{k-1})\\
	&=& (p^k-1)+p^k(p^k-1)-\left((p-1)^2+p^2(p-1)^2+p^4(p-1)^2+\cdots +p^{2(k-1)}(p-1)^2\right)\\
	&=& p^{2k}-1 -(p-1)^2\left(\frac{p^{2k}-1}{p^2-1}\right)\\
\sum\limits_{x\in \mathbb{Z}_{p^k}} deg(x)	&=& \frac{2p^{2k}-2}{p+1}.
\end{eqnarray*}
\end{proof}
\normalfont
\begin{corollary}\label{9}
		$\sum\limits_{x\in \mathbb{Z}_{2^k}}	deg(x)$ $=$ $\frac{ 2^{2k+1}-2}{3}$ for each positive integer $k$.
\end{corollary}
\begin{corollary}\label{10}
	$\sum\limits_{x\in\mathbb{Z}_{p^k}} o(x) = \frac{p^{2k+1}+1}{p+1}$ for each positive integer $k$.
\end{corollary}
\begin{proof}
Refer to \cite[Lemma 2.1]{MN}.
\end{proof}
\begin{corollary}\label{11}
		$\sum\limits_{x\in\mathbb{Z}_{2^k}} o(x)=\frac{2^{2k+1}+1}{3}$ for each positive integer $k$.
\end{corollary}
\begin{corollary}\label{12}
	$1+\sum\limits_{x\in \mathbb{Z}_{2^k}}deg(x)$ $=$ $\sum\limits_{x\in \mathbb{Z}_{2^k}} o(x)$ for each positive integer $k$.
\end{corollary}
\begin{proof}
Apply Corollary \ref{9} and Corollary \ref{11}.
\end{proof}

\noindent Next, we aim to determine the size of the order-divisor graph using the Handshaking Lemma. We start with an example to illustrate this.\\

\begin{example}
We have the  Table \ref{t1} which gives the degree of each vertex in $OD(\mathbb{Z}_6)$.
Using  Table \ref{t1} and the HandShaking Lemma,  we can easily compute the size of $OD(\mathbb{Z}_6)$, as follows:
\begin{eqnarray*}
		Size ~ of ~ OD(\mathbb{Z}_6) &=& \frac{5 + 4 +3 + 3 +3 +4}{2}\\
		&=& 11.
\end{eqnarray*}
By using the same technique as above, we determine the general formula to find the size of $OD(\mathbb{Z}_n)$.
\end{example}

\begin{theorem}\label{13}
The  size of $OD\big(\mathbb{Z}_n\big)$ is given by

$|E\big(OD(\mathbb{Z}_n)\big)|=\frac{1}{2}\left(\sum\limits_{m \rvert n}\bigg(m+\sum\limits_{\lambda|\frac{n}{m}}\phi\big
(\lambda m\big)-2\phi\big(m\big)\bigg)\phi\big(m\big)\right)$, where $\phi$ represents the Euler's totient function.
\end{theorem}
\begin{proof}
	We know that for each divisor $d$ of $n$, $\mathbb{Z}_n$ contains exactly $\phi(d)$ elements that have order $d$. Then apply Lemma \ref{4} and the well known Handshaking Lemma.
\end{proof}
\begin{corollary}\label{14}
For every prime number $p$ and  each positive integer $k$, the size of $OD(\mathbb{Z}_{p^k})$ is given by

$|E\big(OD(\mathbb{Z}_{p^k})\big)|=\frac{p^{2k}-1}{p+1}$.
\end{corollary}
\begin{proof}
	Apply Corollary \ref {7} and the well known Handshaking Lemma.
\end{proof}

\subsection{The Degree and Size within $OD(D_n)$ for the Dihedral Group $D_n=<a,b \mid a^n=b^2=(ab)^2=e>$ }\,

\noindent We describe here the  degree of vertices
and size for the order-divisor graph of the dihedral group. We initiate our
discussion with the aid of an example.

\begin{example}
Consider the order-divisor graphs $OD(\mathcal{D}_4)$ and $OD(\mathcal{D}_5)$ shown in Figure \ref{f2} and Figure \ref{f3} respectively.
We compute directly the degree of each vertex in these graphs as shown in Table \ref{t2} and Table \ref{t3}.
\begin{center}
\begin{figure}[!h]
\includegraphics[width=0.27\textwidth]{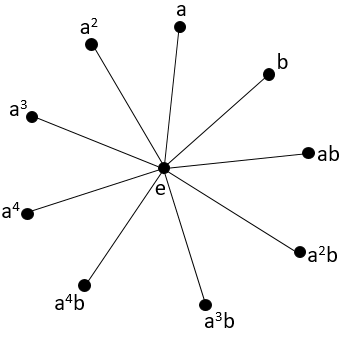}
\caption{OD($\mathcal{D}_5$)}
\label{f3}
\end{figure}
\begin{figure}[!h]
\includegraphics[width=0.27\textwidth]{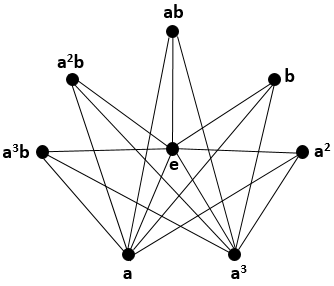}
\caption{OD($\mathcal{D}_4$)}
\label{f4}
\end{figure}
\end{center}

\begin{center}
\begin{table}[!h]
\begin{tabular}{|c|r|c|}
      			\hline
      			vertex & order & degree \\
      			\hline
      			$e$ & $1$ & $7$ \\
      			\hline
      			$a$ & $4$ & $6$\\
      			\hline
      			$a^2$ & $2$ & $3$\\
      			\hline
      			$a^3$ & $4$ & $6$\\
      			\hline
      			$b$ & $2$ & $3$\\
      			\hline
      			$ab$ & $2$ & $3$\\
      			\hline
      			$a^2b$ & $2$ & $3$\\
      			\hline
      			$a^3b$ & $2$ & $3$\\
      			\hline
\end{tabular}
\caption{Degree $\&$ Order for $OD(\mathcal{D}_4)$}
\label{t2}
\end{table}
\begin{table}[!h]
\begin{tabular}{|c|r|c|}
			\hline
			vertex & order & degree \\
			\hline
			$e$ & $1$ & $9$ \\
			\hline
			$a$ & $5$ & $1$\\
			\hline
			$a^2$ & $5$ & $1$\\
			\hline
			$a^3$ & $5$ & $1$\\
			\hline
			$a^4$&$5$&$1$\\
			\hline
			$ab$&$2$&$1$\\
			\hline
			$a^2b$ & $2$ & $1$\\
			\hline
			$a^3b$ & $2$ & $1$\\
			\hline
			$a^4b$ & $2$ & $1$\\
			\hline
			$b$ & $2$ & $1$\\
			\hline
\end{tabular}
\caption{Degree $\&$ Order for $OD(\mathcal{D}_5)$}
\label{t3}
\end{table}
\end{center}
\end{example}

\begin{remark}
It is quite difficult to compute directly the degree of each vertex in $OD(D_n)$ for the large
values of $n$. Hence there is a need to find some explicit formula to determine the degree without actually drawing $O(D_n)$. In this section,
we introduce a general formula for calculating the degree of each vertex  in $O(D_n)$.

\end{remark}
\begin{theorem}\label{5}
	Let $n$ be an odd positive integer with $n \geq 3$ and let  $x \in D_n$ has order $m$. Then\\
		\[deg(x)=
	\begin{cases}
		m - 2\phi(m) + \sum\limits_{\lambda \rvert \frac{n}{m}}\phi(\lambda m), & \text{if $m\textgreater 2$} \\
		2n-1, & \text{if $m = 1$}\\
		1, & \text{if $m = 2$},
	\end{cases}\]
where $\phi$ represents the Euler's totient function.
\end{theorem}
\begin{proof}
When $m$ $> 2$, then $x \in <a>$.   Hence,
\begin{eqnarray*}
	Ldeg(x) &=& \sum\limits_{d \rvert m}\phi(d) - \phi(m)\\
	&=& m - \phi(m).
\end{eqnarray*}	
and
\begin{eqnarray*}
	Udeg(x) &=& \sum\limits_{\lambda m \rvert n}\phi(\lambda m) - \phi(m)\\
	&=& \sum\limits_{\lambda \rvert \frac{n}{m}}\phi(\lambda m) - \phi(m).
\end{eqnarray*}
Hence,
\begin{eqnarray*}
	deg(x) &=& Ldeg(x) + Udeg(x)\\
	&=& m - \phi(m) + \sum\limits_{\lambda \rvert \frac{n}{m}}\phi(\lambda m) - \phi(m)\\
	deg(x)	&=& m - 2\phi(m) + \sum\limits_{\lambda \rvert \frac{n}{m}}\phi(\lambda m).
\end{eqnarray*}
When $m=1$, then $x$ is the identity element and so $deg(x)=2n-1$.  Also when $m=2$, then since $n$ is odd, the identity is the only vertex adjacent with $x$.
Hence, in this case,
$deg(x)=1$.
\end{proof}

\begin{theorem}\label{8}
Let $n$ be an even positive integer with $n \geq 3$ and let  $x \in D_n$ has order $m$. Then
\[deg(x)=
\begin{cases}
m - 2\phi(m) + \sum\limits_{\lambda \rvert \frac{n}{m}}\phi(\lambda m), & \text{if $m$ is odd}; \\
\sum\limits_{\lambda \rvert \frac{n}{2}}\phi(2\lambda), & \text{if $m = 2$}; \\
n+m - 2\phi(m) + \sum\limits_{\lambda \rvert \frac{n}{m}}\phi(\lambda m) 	, & \text{if $m$ is even, $m \textgreater 2$};\\
2n-1, & \text{if $m = 1$},
\end{cases}\]
where $\phi$ represents the Euler's totient function.
\end{theorem}

\begin{proof}
Suppose $m$ is odd. Then $x  \in <a>=\{e,a,a^2,...,a^{n-1}\}$ and $x$ cannot be adjacent to any vertex in $\{b,ab,a^2b,...,a^{n-1}b\}$.
 Hence \begin{eqnarray*}
	Ldeg(x) &=& \sum\limits_{d \rvert m}\phi(d) - \phi(m)\\
	&=& m - \phi(m).
\end{eqnarray*}	
and
\begin{eqnarray*}
	Udeg(x) &=& \sum\limits_{\lambda m \rvert n}\phi(\lambda m) - \phi(m)\\
	&=& \sum\limits_{\lambda \rvert \frac{n}{m}}\phi(\lambda m) - \phi(m).
\end{eqnarray*}
Hence,
\begin{eqnarray*}
	deg(x) &=& Ldeg(x) + Udeg(x)\\
	&=& m - \phi(m) + \sum\limits_{\lambda \rvert \frac{n}{m}}\phi(\lambda m) - \phi(m)\\
	deg(x)	&=& m - 2\phi(m) + \sum\limits_{\lambda \rvert \frac{n}{m}}\phi(\lambda m).
\end{eqnarray*}

\noindent Suppose $m=2$. Then $x$ cannot be adjacent to any  vertx in $\{b,ab,ab^2,...,ab^{n-1}\}$. Hence $Ldeg(x)=1$. Thus

\begin{eqnarray*}
	deg(x) &=& Ldeg(x) + Udeg(x)\\
	&=& 1 + \sum\limits_{\lambda \rvert \frac{n}{m}}\phi(\lambda m) - 1=\sum\limits_{\lambda \rvert \frac{n}{m}}\phi(\lambda m).
\end{eqnarray*}

\noindent Suppose $m>2$ is even. Then $x \in <a>=\{e,a,a^2,..,a^{n-1}\}$ and all the vertices in $\{b,ab,a^2b,..,a^{n-1}b\}$ are lower vertices  adjacent to $x$.

Hence \begin{eqnarray*}
	Ldeg(x) &=& n+\sum\limits_{d \rvert m}\phi(d) - \phi(m)\\
	&=& n+m - \phi(m).
\end{eqnarray*}	
and
\begin{eqnarray*}
	Udeg(x) &=& \sum\limits_{\lambda m \rvert n}\phi(\lambda m) - \phi(m)\\
	&=& \sum\limits_{\lambda \rvert \frac{n}{m}}\phi(\lambda m) - \phi(m).
\end{eqnarray*}
Hence,
\begin{eqnarray*}
	deg(x) &=& Ldeg(x) + Udeg(x)\\
	&=& n+m - \phi(m) + \sum\limits_{\lambda \rvert \frac{n}{m}}\phi(\lambda m) - \phi(m)\\
	deg(x)	&=& n+m - 2\phi(m) + \sum\limits_{\lambda \rvert \frac{n}{m}}\phi(\lambda m).
\end{eqnarray*}
When $m=1$, then $x$ is the identity element and so $deg(x)=2n-1$.
\end{proof}

\begin{theorem}\label{15}
Let $n$ be an odd  integer with $n \geq 3$. Then
the size  of  $OD\big(D_n\big)$ is given by

$|E(OD(D_n))|=\frac{1}{2}\left(3n-1+\sum\limits_{\substack {m \mid n \\m > 1}}\bigg(m-
2\phi\big(m\big)+\sum\limits_{\lambda|\frac{n}{m}}\phi\big(\lambda m\big)\bigg)\phi\big(m\big)\right)$, where $\phi$ represents the Euler's totient function.
\end{theorem}
\begin{proof}
Since, $n$ is odd, the elements with order $2$ are $\{b,ab,a^2b,...,a^{n-1}b\}$.
Moreover in the cyclic part $<a>$ of $D_n$, for each divisor $m$ of $n$,
there are exactly $\phi(m)$ elements of order $m$.
Then apply Theorem \ref{5} and the well known Handshaking Lemma.
\end{proof}

\begin{theorem}\label{16}
Let $n$  be an even  integer with $\geq 3$. Then
the size  of  $OD\big(D_n\big)$ is given by
\begin{eqnarray*}
|E(OD(D_n))|=\frac{1}{2}\left( \begin{aligned}& 2n-1+\big(n+1\big)\sum\limits_{\lambda \rvert \frac{n}{2}}
\phi(2 \lambda)+\sum\limits_{\substack {m > 1, m \mid n\\ m~is~odd}}\bigg(m-2\phi\big(m\big)+
\sum\limits_{\lambda|\frac{n}{m}}\phi\big(\lambda m\big)\bigg)\phi\big(m\big) \\ &+ \sum\limits_{\substack {m > 2, m \mid n\\ m~is~even}}\bigg(n+m-
2\phi\big(m\big)+\sum\limits_{\lambda|\frac{n}{m}}\phi\big(\lambda m\big)\bigg)\phi\big(m\big) \end{aligned}\right),
\end{eqnarray*}
where $\phi$ represents the Euler's totient function.
\end{theorem}

\begin{proof}
Since, $n$ is even, so the elements with order $2$ are  $\{a^{n/2}, b,ab,a^2b,...,a^{n-1}b\}$.
Moreover in the cyclic part $<a>$ of $D_n$, for each divisor $m$ of $n$,
there are exactly $\phi(m)$ elements of order $m$.
 Then apply Theorem \ref{8} and the well known Handshaking Lemma.
\end{proof}

\section{PARI/GP codes to compute degree and size in order-divisor graphs of $\mathbb{Z}_n$ and $\mathcal{D}_n$}
 \noindent We can employ the computer algebra software PARI/GP \cite{P} to compute the degree of each vertex and the size within the order-divisor graph
  associated with the cyclic group $\mathbb{Z}_n$ and  the dihedral group $\mathcal{D}_n$. The successful verification of these codes
 is illustrated by the PARI/GP sessions conducted in a Windows environment, as depicted in Figures \ref{f5},  \ref{f6}
 \ref{f7}, \ref{f8},\ref{f10}, and \ref{f11}.

 \bigskip

\noindent  \hrulefill PARI/GP code to find degree of the vertex $x$ in $OD(\mathbb{Z}_n)$ \hrulefill

\begin{mdframed}
\begin{equation*}
\begin{aligned}
d(x,n)&=n/gcd(x,n) - 2*eulerphi(n/gcd(x,n))\\& + sumdiv(n/(n/gcd(x,n)), t, eulerphi(n/gcd(x,n)*t))
\end{aligned}
\end{equation*}
\end{mdframed}

\begin{center}
\begin{figure}[!h]
\includegraphics[width= 1\textwidth]{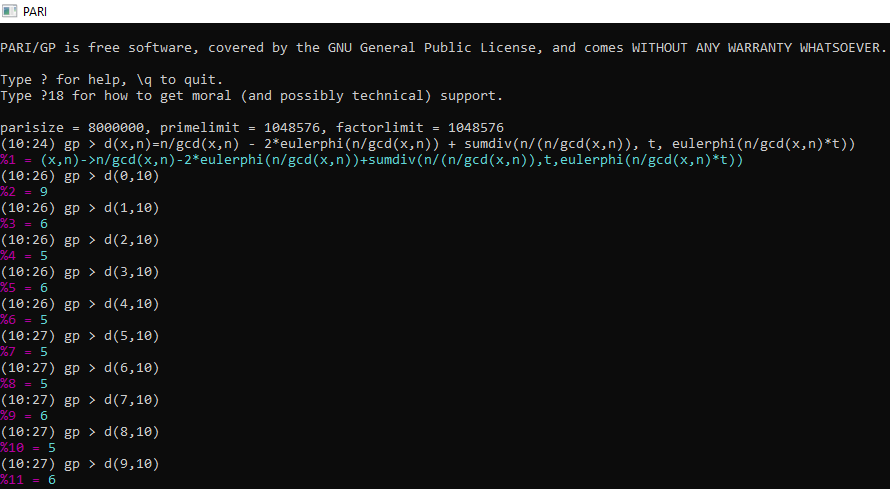}
\caption{PARI/GP session for computing degree in $OD(\mathbb{Z}_{10})$  under Windows.}
\label{f5}
\end{figure}
\end{center}

\noindent  \hrulefill PARI/GP code to find the size of  $OD(\mathbb{Z}_n)$ \hrulefill

\begin{mdframed}
\begin{equation*}
\begin{aligned}
s(n)&=(sumdiv(n, m, (m - 2*eulerphi(m) \\ &+ sumdiv(n/m, t, eulerphi(m*t)))*eulerphi(m)))/2
\end{aligned}
\end{equation*}
\end{mdframed}

\begin{center}
\begin{figure}[!h]
\includegraphics[width=0.7\textwidth]{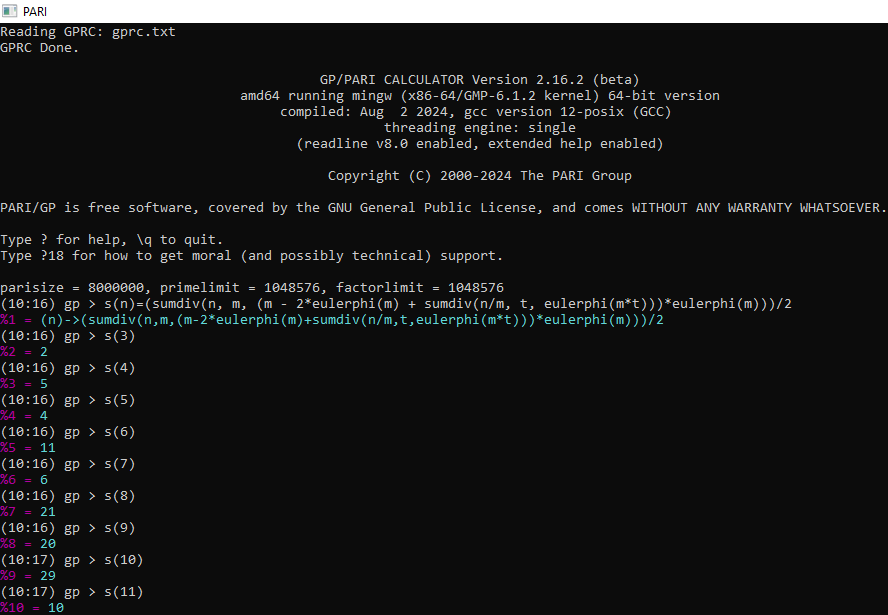}
\caption{PARI/GP session for computing size of $OD(\mathbb{Z}_n)$ under Windows.}
\label{f6}
\end{figure}
\end{center}

\noindent  \hrulefill PARI/GP code to find degree of the vertex $x$ of order $m$ in $OD(\mathcal{D}_n)$, where $n \geq 3$ is odd.\hrulefill

\begin{mdframed}
\begin{equation*}
\begin{aligned}
d(n,m)&=if(m>2, print(m-2*eulerphi(m) + sumdiv(n/m, t, eulerphi(m*t)),''~''),\\ & if(m==1, print(2*n-1,''~''), \\& if(m==2, print(1,''~''))))
\end{aligned}
\end{equation*}
\end{mdframed}

\begin{center}
\begin{figure}[!h]
\includegraphics[width=0.9\textwidth]{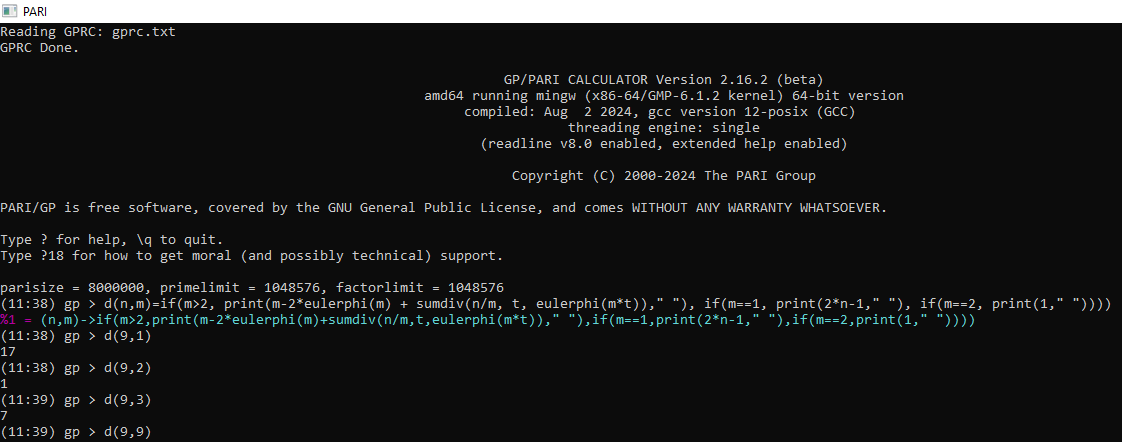}
\caption{PARI/GP session for computing degree in $OD(\mathcal{D}_9)$  under Windows.}
\label{f7}
\end{figure}
\end{center}

\medskip
\noindent  \hrulefill PARI/GP code to find degree of the vertex $x$ of order $m$ in $OD(\mathcal{D}_n)$, where $n \geq 3$ is even.\hrulefill

\begin{mdframed}
\begin{equation*}
\begin{aligned}
d(n,m)&=if(m\%2==1 ~\& \& ~m >1, print(m-2*eulerphi(m) + sumdiv(n/m, t, eulerphi(m*t)), ''~''),\\
& if (m==2, print(sumdiv(n/m, t, eulerphi(m*t)), ''~'' ),\\
& if (m\%2==0 ~\& \&~ m >>2, print(n+m-2*eulerphi(m)+sumdiv(n/m, t, eulerphi(m*t)), ''~''), \\
& if(m==1, print(2*n-1,''~'')))))
\end{aligned}
\end{equation*}
\end{mdframed}

\begin{center}
\begin{figure}[!h]
\includegraphics[width=0.7\textwidth]{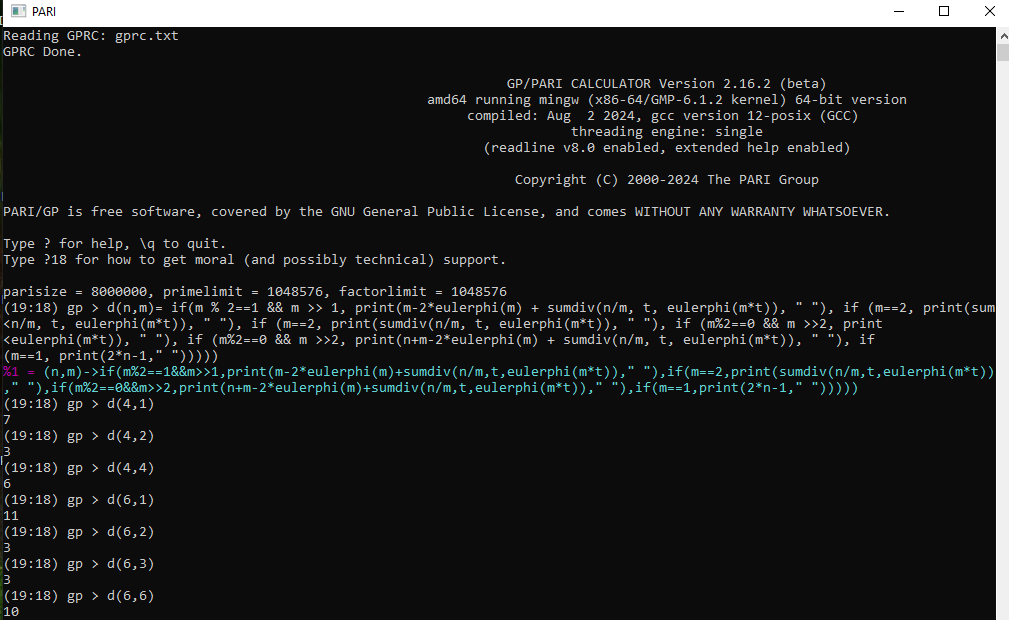}
\caption{PARI/GP session for computing degree in $OD(\mathcal{D}_4)$ and $OD(\mathcal{D}_6)$  under Windows.}
\label{f8}
\end{figure}
\end{center}

\noindent  \hrulefill PARI/GP code to find the size of  $OD(\mathcal{D}_n)$, where $n \geq 3$ is odd. \hrulefill

\begin{mdframed}
\begin{equation*}
\begin{aligned}
sd(n)&=(3*n-1+sumdiv(n, m, if (m>>1,  (m - 2*eulerphi(m) \\&+ sumdiv(n/m, t, eulerphi(m*t))))*eulerphi(m)))/2
\end{aligned}
\end{equation*}
\end{mdframed}

\begin{center}
\begin{figure}[!h]
\includegraphics[width=0.8\textwidth]{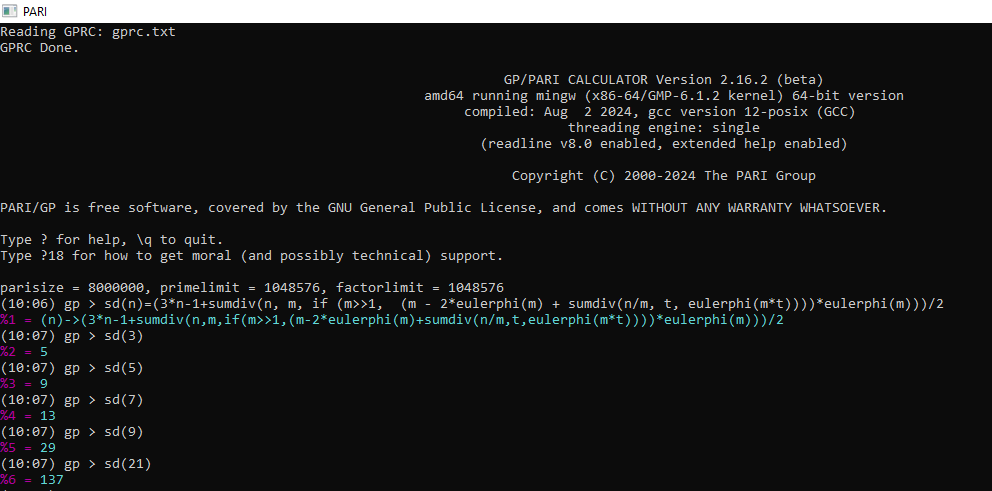}
\caption{PARI/GP session for computing size of $OD(\mathcal{D}_n)$, where $n$ is odd (under Windows).}
\label{f10}
\end{figure}
\end{center}

\noindent  \hrulefill PARI/GP code to find the size of  $OD(\mathcal{D}_n)$, where $n \geq 3$ is even. \hrulefill
\begin{mdframed}
\begin{equation*}
\begin{aligned}
sd(n)&=(2*n-1+(n+1)*sumdiv(n/2,t,eulerphi(2*t))\\
&+sumdiv(n,m,if(m>>1 ~\&\& ~m\%2==1,(m-2*eulerphi(m)\\
&+sumdiv(n/m,t,eulerphi(m*t))))*eulerphi(m))\\
&+sumdiv(n,m,if(m>>2~\&\&~m\%2==0,(n+m-2*eulerphi(m)\\
&+sumdiv(n/m,t,eulerphi(m*t))))*eulerphi(m)))/2
\end{aligned}
\end{equation*}
\end{mdframed}

\begin{center}
\begin{figure}[!h]
\includegraphics[width=0.8\textwidth]{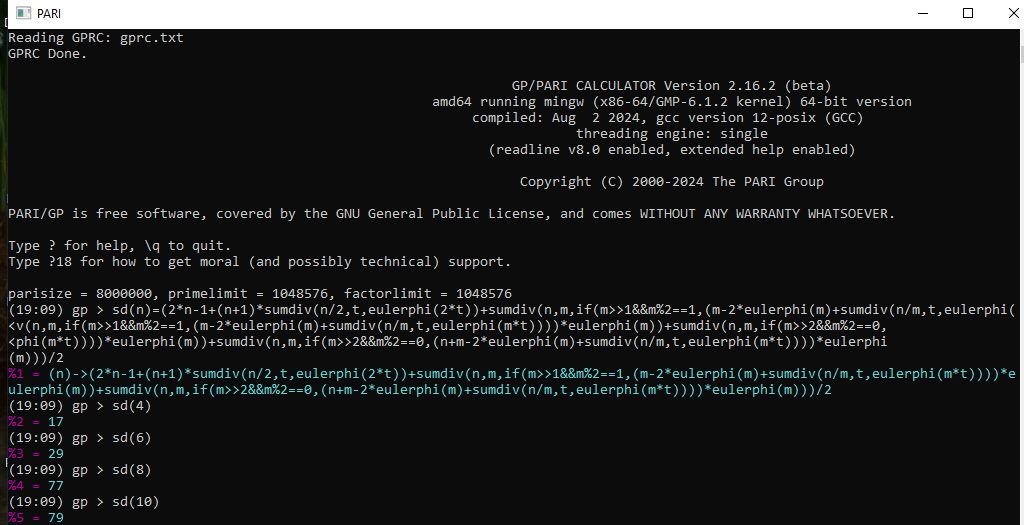}
\caption{PARI/GP session for computing size of $OD(\mathcal{D}_n)$, where $n$ is even (under Windows).}
\label{f11}
\end{figure}
\end{center}

\textbf{Acknowledgement:}

This research is supported
by HEC( Higher Education Commission) of Pakistan under NRPU Project No:14522.

$\\$

\begin{thebibliography}{11}

\bibitem{AAM}{A. Abdollahi, S. Akbari, and H. R. Maimani,} {Non-commuting graph of a group,} J. Algebra, {\bf 298} (2006), 468-492.

\bibitem{AHM}{ S. Akbari, F. Heydari, and M. Maghasedi,}
{The intersection graph of a group,}
J. Alg.  Appl., {\bf 14}(05) (2015), 1550065.

\bibitem{AMRR}{S. Akbari, A. Mohammadian, H. Radjavi, and P. Raja,}
{On the diameter of commuting graphs,}
Linear Algebra Appl., {\bf 418}(1) (2006), 161-176.



\bibitem{ASH}{F. Ali, M. Salman, and  S. Huang,} {On the commuting graph of Dihedral group,}   Comm.  Algebra, {\bf 44}(6) (2016), 2389-2401.


\bibitem{B}{A. Ballester-Bolinches and J. Cossey,} {Graphs and classes of finite groups,} {Note di Mathematica,} {\bf 33}(1) (2013), 89-94.

\bibitem{BBPR}{C. Bates, D. Bundy, S. Perkins, and P. Rowley,}
{Commuting involution graphs for symmetric groups,}
J. Algebra, {\bf 266} (2003), 133-153.

\bibitem{BHM}{E. A. Bertram, M. Herzog, and A. Mann, }
{On a graph related to cojugacy classes of groups,} Bull. London Math. Soc., {\bf 22} (1990), 569-575.





\bibitem{FB}{F. Budden,} {Cayley graphs for some well-known groups,}
The Mathematical Gazette,  {\bf 69} (1985), 271-278.

\bibitem{BC1} {T. C. Burness, E. Covato,} {On the prime graph of simple groups,}
Bull. Aust. Math. Soc.,  {\bf 91} (2015), 227-240.



\bibitem{C1} {P. J. Cameron,} {The power graph of a finite group II,}
 J. Group Theory, {\bf 13}(2010), 779-783.

\bibitem{C2}{P. J. Cameron, S. Ghosh,}
{The power graph of a finite group,} {Discrete Mathematics,} {\bf 311} (2011), 1220-1222.


\bibitem{F}{R. Frucht,} {Graphs of degree $3$ with given abstract group,} {Canad. J. Math.,} {\bf 1} (1949), 365-378.

\bibitem{G}{J. A. Gallians,} {\em Contemporary Abstract Algebra,} {8th edition, Brooks/Cole,} (2013).

\bibitem{P}{ The PARI Group,} { PARI/GP (version 2.16.2),} \url{http://pari.math.u-bordeaux.fr}(2024).



\bibitem{R}{Z. Raza,} {Commuting graphs of dihedral type groups,} {Applied Mathematics E Notes} {\bf 13}(2013),  221-227.

\bibitem{RBIK}{S. U. Rehman, A. Q. Baig, M. Imran, and Z. U. Khan,}
{Order divisor graphs of finite groups,} {Analele Stiintifice ale Universitatii Ovidius Constanta, Seria Matematica,} {\bf 26}(3) (2018), 29-40.


\bibitem{R1}{ S. U. Rehman,  G. Farid, T. Tariq, E,  Bonyah,} {Equal-Square graphs sssociated with finite groups,} {Journal of Mathematics,} 2022.










\bibitem{MN}{N. Mansuro$\check{g}$lu,}
{Some properties on sums of element orders in finite p-groups,} Mathematical Sciences and Applications E-Notes, 8(1), 51-54.


		
		

		

\bibitem{K}{B. Khosravi,} {On the prime graph of a finite group,} Groups-St. Andrews (Bath, 2009), {\bf 2} 424-428,
London Math. Soc. Lecture Note Series, {\bf 388}, Cambridge University Press, 2011.

\bibitem{Lu}{ M. S. Lucido,} { The diameter of the prime graph of a finite group,}
J. Group Theory, {\bf 2} (1999), 157-172.

\bibitem{L} {C. H. Li,} {Isomorphisms and classification of Cayley graphs of small valencies on
finite abelian groups,} Australas. J. Combin., {\bf 12} (1995), 3-14.
\bibitem{L2} {C. H.  Li,} {On isomorphisms of finite Cayley graphs-a survey,} Discrete Mathematics, { \bf 256}
(2002), 301-334.

\bibitem{MPA} {M. Mirzargar, P. P. Pach, and A. R. Ashrafi,} {The automorphism group of commuting graph of a finite group,} Bull. Korean Math. Soc.,
 {\bf 51}(4) (2014), 1145-1153.


\bibitem{MAN} {M. Mirzargar, A. R. Ashrafi, and M. J. Nadjafi-Arani,}
{On the power graph of a finite group,} Filomat, {\bf 26}(6) (2012), 1201-1208.
		
		
\bibitem{N}{B. H. Neumann,} {A problem of Paul Erd$\ddot{o}$s on group,}
J. Aust. Math. Soc., {\bf 21}(1976), 467-472.


%
\bibitem{P} {C. W. Parker,} {The commuting graph of a soluble group,} Bull. London Math. Soc., {\bf 45}(2013), 839-848.

		\bibitem{KHR}{Kenneth H. Rosen,}
		{\em Elementry Number Theory and its Application,} 5th edition,
		Addison-Wesley Publishing company,  2005, USA.
		
		\bibitem{R}{Z. Raza,}
		{Commuting graphs of dihedral type groups,} {Applied Mathematics E-Notes,} {\bf 13}(2013), 221-227.
		
		
		
		

		
\bibitem{W} {D. B. West,} {\em Introduction to Graph Theory,} Prentice Hall. Inc. Upper Saddle River, NJ, 1996.
		
\bibitem{ZA}{Y. F. Zakariya and M. R. Alfuraidan,}
		{Inverse graphs associated with finite groups,} Electronic Journal of Graph Theory and Applications, {\bf 5}(1) (2017), 142-154.

\end{thebibliography}
\end{document}